\documentclass[10pt]{amsart}
\usepackage[utf8]{inputenc}
\usepackage{amssymb}
\usepackage{amscd}
\usepackage{amsfonts}
\usepackage{amsmath}
\usepackage{amsthm}
\usepackage[all]{xy}
\usepackage{mathrsfs}
\usepackage[noadjust]{cite}
\usepackage[colorlinks=true, allcolors=blue]{hyperref}
\usepackage[shortlabels]{enumitem}
\usepackage{tikz}

\usepackage{bm}
\usepackage{comment}

\usepackage[mathscr]{eucal}

\newtheorem{theorem}{Theorem}[section]
\newtheorem*{theorem*}{Theorem}
\newtheorem{lemma}[theorem]{Lemma}
\newtheorem{proposition}[theorem]{Proposition}
\newtheorem{corollary}[theorem]{Corollary}

\theoremstyle{definition}
\newtheorem{definition}[theorem]{Definition}

\newtheorem{remark}[theorem]{Remark}

\newcommand{\pcoor}[1]{%
  \begingroup\lccode`~=`: \lowercase{\endgroup
  \edef~}{\mathbin{\mathchar\the\mathcode`:}\nobreak}%
  [
  \begingroup
  \mathcode`:=\string"8000
  #1%
  \endgroup 
  ]
} 

\newcommand{\tens}[1]{%
  \mathbin{\mathop{\otimes}\displaylimits_{#1}}%
} 

\newcommand{\sh}[1]{\mathcal{#1}}

\newcommand{\msp}[1]{{#1}}
\newcommand{\orbi}[1]{\mathcal{#1}}

\def\Spec{\mathrm{Spec}}
\def\Hom{\mathrm{Hom}}

\def\GL{\mathrm{GL}}

\def\vir{\mathrm{vir}}
\def\Hilb{\mathrm{Hilb}}
\def\Proj{\mathrm{Proj}}
\def\Coh{\mathrm{Coh}}
\def\Mod{\mathrm{Mod}}

\def\Sym{\mathrm{Sym}}
\def\Ext{\mathrm{Ext}}
\def\Jac{\mathrm{Jac}}
\def\Tr{\mathrm{Tr}}

\def\Quot{\mathrm{Quot}}
\def\supp{\mathrm{supp}}
\def\End{\mathrm{End}}

\def\ConSch{\mathcal{C}}

\def\ss{\mathrm{ss}}
\def\fd{\mathrm{fd}}
\def\simple{\mathrm{sp}}

\def\tt{\bm{t}}
\def\dd{\mathbf{d}}
\def\ff{\bm{f}}

\def\ww{\mathbf{w}}
\def\PL{\mathrm{PL}}

\newcommand{\sslash}{\mathbin{/\mkern-6mu/}}

\def\planep{\mathcal{P}} 

\title{Donaldson--Thomas theory of quantum Fermat quintic threefolds II}

\author{Yu-Hsiang Liu}
\address{Department of Mathematics, University of British Columbia}
\email{\href{mailto:yuliu@math.ubc.ca}{yuliu@math.ubc.ca}}
\thanks{} 

\begin{document}

\begin{abstract} 
This paper is a continuation of author's previous work, where we defined Donaldson--Thomas invariants of quantum Fermat threefolds. In this paper, we study the generic quantum Fermat threefold. We give explicit local models for Hilbert schemes of points as quivers with potential, and compute degree zero Donaldson--Thomas invariants. The result is expressed in terms of certain colored plane partitions.
\end{abstract}

\maketitle

\setcounter{tocdepth}{1}

\tableofcontents

\section{Introduction}

Donaldson--Thomas (DT) invariants, first introduced by Thomas \cite{Tho00}, are integer-valued deformation invariants on a compact Calabi--Yau $3$-fold. It provides a virtual count of curves embedded into a Calabi--Yau $3$-fold, and is conjecturally equivalent to other enumerative invariants, such as Gromov--Witten invariants \cite{MNOP1}, Pandharipande--Thomas invariants \cite{PT09}, and Gopakumar--Vafa invariants \cite{MT18}. All of which are closely related to the BPS invariants, which are motivated by counting BPS states in the string theory.

A quantum Fermat quintic threefold, introduced in \cite{Kan15}, is the Fermat quintic hypersurface in a quantum projective $4$-space, using the language of non-commutative projective schemes developed by Artin and Zhang \cite{AZ94}. DT invariants of a quantum Fermat quintic threefold are defined in the author's previous work \cite{Liu19} as follows. Any quantum Fermat quintic threefold is represented by a coherent sheaf $\sh{A}$ of non-commutative $\sh{O}_X$-algebras on $X\cong\mathbb{P}^3$. We consider Simpson's Hilbert schemes $\Hilb^h(\sh{A})$ parameterizing $\sh{A}$-module quotients $\sh{A}\to\sh{F}$. It is shown that if $\deg(h)\leq 1$, then $\Hilb^h(\sh{A})$ admits a symmetric obstruction theory, thus carries a virtual fundamental class $[\Hilb^h(\sh{A})]_{\vir}$ of virtual dimension $0$. DT invariants are defined as
\[
\int_{[\Hilb^h(\sh{A})]_{\vir}} 1,
\]
which equals to the Euler characteristic $\chi\big(\Hilb^h(\sh{A}),\nu_{\Hilb^h(\sh{A})}\big)$ weighted by the Behrend function \cite{Beh09}. We will only consider constant Hilbert polynomials, and write
\[
Z^{\sh{A}}(t)=\sum_{n=0}^{\infty}\chi\big(\Hilb^n(\sh{A}),\nu_{\Hilb^n(\sh{A})}\big)\,t^n
\]
for the generating function of degree zero DT invariants.

\vspace{1mm}

In this paper, we will focus on the generic quantum Fermat quintic threefold \cite[\textsection 2.3]{Liu19}. We first give explicit local models of the generic quantum Fermat quintic threefold $(X,\sh{A})$.

\begin{theorem}[= Theorem~\ref{thm:local-model}]
There exist a stratification $X=X_{(0)}\coprod\ldots\coprod X_{(3)}$ of $X$ and coherent sheaves $\sh{J}_{(i)}$ of non-commutative algebras on $\mathbb{C}^3$ such that for any point $p\in X_{(i)}$, there is an \emph{analytic} local chart $U\to\mathbb{C}^3$ of $p$ with a (non-unique) isomorphism
\[
\sh{A}|_{U}\cong\sh{J}_{(i)}|_{U}
\]
of sheaves of non-commutative algebras. These sheaves $\sh{J}_{(i)}$'s of algebras are (up to Morita equivalence) Jacobi algebras of quivers with potential.
\end{theorem}

More specifically, $\sh{J}_{(i)}$'s are (up to Morita equivalence) just copies of $\mathbb{C}[x,y,z]$  for $i\neq 0$, whose DT invariants are well-studied. The Jacobi algebra $\sh{J}_{(0)}$ is defined by the quiver $Q$
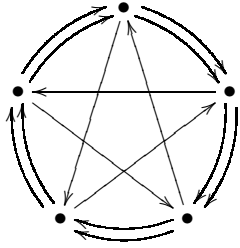
\[
\begin{xy} 0;<2pt,0pt>:<0pt,-2pt>:: 
(40,16) *+{\bullet} ="0",
(8,40) *+{\bullet} ="2",
(32,40) *+{\bullet} ="1",
(0,16) *+{\bullet} ="3",
(20,0) *+{\bullet} ="4",
(20,20) *+<2pt>{},
"4":{\ar@<-.5ex>@/^/"0"},
"4":{\ar@<.5ex>@/^/"0"},
"4":{\ar"2"},
"0":{\ar@<-.5ex>@/^/"1"},
"0":{\ar@<.5ex>@/^/"1"},
"0":{\ar"3"},
"1":{\ar@<-.5ex>@/^/"2"},
"1":{\ar@<.5ex>@/^/"2"},
"1":{\ar"4"},
"2":{\ar@<-.5ex>@/^/"3"},
"2":{\ar@<.5ex>@/^/"3"},
"2":{\ar"0"},
"3":{\ar@<-.5ex>@/^/"4"},
"3":{\ar@<.5ex>@/^/"4"},
"3":{\ar"1"},
\end{xy}
\]
with certain potential $W$.

\vspace{1mm}

These local models $\sh{J}_{(i)}$'s allows us to stratify Hilbert schemes $\Hilb^n(\sh{A})$ of points. The generating function $Z^{\sh{A}}(t)$ can be expressed in terms of Hilbert schemes $\Hilb^n(\sh{J}_{(i)})$.

\begin{theorem}[= Theorem~\ref{thm:fibration}]
We have
\[
Z^{\sh{A}}(t)=\prod_{i=0}^3\left(\sum_{n=0}^{\infty}\chi\big(\Hilb^n(\sh{J}_{(i)})_0,\nu_{\Hilb^n(\sh{J}_{(i)})}\big)\,t^n \right),
\]
where $\Hilb^n(\sh{J}_{(i)})_0$ can be regarded as an analogue of punctual Hilbert scheme of points.
\end{theorem}

Combing with the known result on DT invariants of $\mathbb{C}^3$, it eventually leads to
\[
Z^{\sh{A}}(t)=\Big(Z^{Q,W}(t)\Big)^{10}\Big(M(-t^5)\Big)^{-50},
\]
where $Z^{Q,W}(t)$ is the generating function of DT invariants of the quiver $(Q,W)$ with potential, and $M(t)$ is the MacMahon function.

\vspace{1mm}

Finally, we turn to the computation of the DT invariants $Z^{Q,W}(t)$. We observe that the quiver $Q$ is the McKay quiver of the $\mu_5$-action on $\mathbb{C}^3$ with weight $(1,1,3)$. It associates an orbifold $[\mathbb{C}^3/\mu_5]$, and DT invariants on $[\mathbb{C}^3/\mu_5]$ were computed in \cite{BCY} using the notion of colored plane partitions \cite{You10}. However, our DT invariants $Z^{Q,W}$ use a different framing vector (stability condition). We introduce the notion of \emph{$Q$-multi-colored plane partitions} associated to a quiver $Q$. Each $Q$-multi-colored plane partition associates a dimension vector $\dd$ of $Q$, and we denote by $n_{\dd}(Q)$ the number of $Q$-multi-colored plane partitions with dimension vector $\dd$.

\begin{theorem}[= Theorem~\ref{thm:DT-to-part}]
We have
\[
Z^{Q,W}(t)=\sum_{n=1}^{\infty}\left(\sum_{|\dd|=n}(-1)^{|\dd|+\langle \dd,\dd\rangle_Q}n_{\dd}(Q)\right)t^n,
\]
\end{theorem}

Unfortunately, we do not obtain a closed formula for $Z^{Q,W}(t)$. But we show the number $n_{\dd}(Q)$ can be computed from the numbers of $\mu_5(1,1,3)$-colored plane partitions.

\subsection*{Notations}

We work over the field $\mathbb{C}$ of complex numbers. All schemes or algebras are separated and noetherian over $\mathbb{C}$. All (sheaves of) algebras are associative and unital. 
By ``non-commutative'', we mean not necessarily commutative, and we assume that non-commutative rings are both left and right noetherian. For a (sheaf of) non-commutative ring $A$, an $A$-module is always a left $A$-module. All rings without specified non-commutative are commutative.

\vspace{1mm}

We are particularly interested in a special class of non-commutative rings, \emph{quantum polynomial rings}. They are polynomial rings with variables only commute up to a non-zero scalar. We will use the notation
\[
\mathbb{C}\langle x_1,\ldots, x_n\rangle_{(q_{ij})}:=\mathbb{C}\langle x_1,\ldots, x_n\rangle\big/\left(x_ix_j-q_{ij}x_jx_i \right),
\]
where $q_{ij}\in\mathbb{C}^*$ and $q_{ii}=q_{ij}q_{ji}=1$ for all $i,j$.

\vspace{1mm}

For any scheme $X$ of finite type (over $\mathbb{C}$), we write $\chi(X)$ for the \emph{topological} Euler characteristic of $X$. We also consider the weighted Euler characteristic
\[
\chi_{\vir}(X):=\chi(X,\nu_X)=\sum_{c\in\mathbb{Z}} c\chi(\nu_X^{-1}(c)),
\]
where $\nu_X$ is the Behrend constructible function. For a locally closed subscheme $Z\subset X$, we write
\[
\chi_{\vir}(X,Z)=\chi(Z,\nu_X|_Z).
\]

\vspace{1mm}

We will study various of \emph{generating functions}. We will use the notation $Z(t_1,\ldots,t_n)$, which is an element in the formal power series $\mathbb{Z}[\![t_1,\ldots,t_n]\!]$. The constant term $Z(0,\ldots,0)$ will always be $1$, hence the product between generating functions makes sense.

\vspace{4mm}

Throughout this paper, we only consider the generic quantum Fermat threefold. We will fix a particular graded algebra
\[
\mathbb{C}\langle t_0,\ldots,t_4\rangle_{(q_{ij})}\big/\left( \sum_{k=0}^4 t_k^5\right)
\]
with quantum parameters
\[
(q_{ij})_{i,j}=\begin{pmatrix}
1 & q & q^{-1} & q & q^{-1} \\
q^{-1} & 1 & q & q^{-1} & q \\
q & q^{-1} & 1 & q & q^{-1} \\
q^{-1} & q & q^{-1} & 1 & q \\
q & q^{-1} & q & q^{-1} & 1
\end{pmatrix},
\]
where $q\in\mu_5$ is a fixed primitive root. Let $(X,\sh{A})$ be the associated pair of smooth projective variety $X$ with a coherent sheaf $\sh{A}$ of non-commutative $\sh{O}_X$-algebras on $X$. See \cite[\textsection 3.3]{Liu19} for precise definitions.

\section{Quivers with potential}

We briefly review non-commutative DT theory for a quiver with potential and fix some notations. A standard reference is \cite{Sze08}.

Let $Q=(Q_0,Q_1)$ be a quiver, where $Q_0$ is the set of vertices, and $Q_1$ the set of arrows. Let $\mathbb{C}Q$ be the path algebra. We will denote $N_Q:=\mathbb{Z}^{\oplus Q_0}$ the free abelian group of dimension vectors, and $N_Q^+:=\mathbb{Z}_{\geq 0}^{\oplus Q_0}$. There is a bilinear form on $N$ defined by
\[
\langle \dd,\dd'\rangle_Q=\sum_{i\in Q_0} d_i d'_{i}-\sum_{a\in Q_1} d_{s(a)} d'_{t(a)}.
\]
The Euler pairing is given by
\[
\chi_Q(\dd,\dd')=\langle \dd,\dd'\rangle_Q-\langle \dd',\dd\rangle_Q.
\]

Let $\ff\in N_Q^+$ be a framing vector, $\ff\neq 0$. A $\ff$-framed representation of $Q$ is a representation $V=(V_i,T_a)_{i\in Q_0, a\in Q_1}$ of $Q$ with vectors $v_i^1,\ldots, v_i^{f_i}$ in $V_i$ for each $i$ which generates $V$ (as a left $\mathbb{C}Q$-module).

Let $W$ be a potential of $Q$, a linear combination of cyclic paths. We define the Jacobi algebra
\[
\Jac(Q,W)=\mathbb{C}Q\big/\left(\partial_a W\right)_{a\in Q_1},
\]
A representation of $(Q,W)$ is a finite-dimensional left $\Jac(Q,W)$-module, and one define framed representations of $(Q,W)$ similarly.

It is known that the fine moduli space $M^{\ff,\dd}(Q,W)$ of $\ff$-framed representations of $(Q,W)$ is a critical locus of a regular function on a smooth scheme. Therefore it makes sense to define DT invariants via weighted Euler characteristics. Let
\[
Z^{Q,W,\ff}(\tt)=\sum_{\dd\in N_Q^+}\chi_{\vir}\big(M^{\ff,\dd}(Q,W)\big)\,\tt^{\dd}
\]
be the generating function of DT invariants of $(Q,W)$ with framing vector $\ff$, where $\tt=(t_i)_{i\in Q_0}$ and $\tt^{\dd}=\prod_{i\in Q_0} t_i^{d_i}$. The framing vector $\ff$ should be interpreted as a choice of stability condition.

If the framing vector $\ff=(1,\ldots,1)$, we will simply write $Z^{Q,W}$ for $Z^{Q,W,(1,\ldots,1)}$. In this case, $(1,\ldots,1)$-framed representations are finite-dimensional \emph{cyclic} $\Jac(Q,W)$-modules. We write
\[
\Hilb^{\dd}(Q,W):= M^{(1,\ldots,1),\dd}(Q,W)
\]
which can be viewed as Hilbert schemes of points on the non-commutative affine space $\Jac(Q,W)$. For integer $n$, let
\[
\Hilb^n(Q,W)=\prod_{|\dd|=n}\Hilb^{\dd}(Q,W)
\]
We will abuse the notation and write
\[
Z^{Q,W}(t):=Z^{Q,W}(t,\ldots,t)=\sum_{n=0}^{\infty}\chi_{\vir}\big(\Hilb^n(Q,W)\big)\,t^n.
\]

\section{Simple modules and the Ext-quiver}

In this section, we study zero-dimensional coherent $\sh{A}$-modules, which we will simply call finite-dimensional $\sh{A}$-modules. Let $\Coh(\sh{A})_{\fd}$ be the category of finite-dimensional $\sh{A}$-modules.

\subsection{Finite-dimensional $\sh{A}$-modules}

We first note that the category $\Coh(\sh{A})_{\fd}$ can be studied locally. Any finite-dimensional $\sh{A}$-module is supported (as a coherent sheaf on $X$) in finitely many points.

\begin{lemma}
For any finite-dimensional $\sh{A}$-module $\sh{F}$,
\[
\sh{F}=\bigoplus_{x\in\supp(\sh{F})}\sh{F}_x.
\]
\end{lemma}

\begin{proof}
We have $\sh{F}=\bigoplus_{x\in\supp(\sh{F})}\sh{F}_x$ as coherent sheaves on $X$. Since the $\sh{A}$-action on $\sh{F}$ is defined locally, $\sh{F}_x$ is naturally an $\sh{A}$-submodule for each $x$. Furthermore, it is clear that the projection $\sh{F}\to\sh{F}_x$ is a morphism of $\sh{A}$-modules. We see that $\sh{F}=\bigoplus_{x\in\supp(\sh{F})}\sh{F}_x$ in $\Coh(\sh{A})_{\fd}$.
\end{proof}

If we choose an affine open cover $\{U_i\}$ of $X$, then for each $U_i$, $\sh{A}|_{U_i}$ is given by a non-commutative algebra, and $\Coh(\sh{A}|_{U_i})$ is the category $\Mod(\sh{A}|_{U_i})$ of honest finite-dimensional modules. Then the categories $\{\Coh(\sh{A}|_{U_i})\}_i$ can be regarded as an affine open cover of $\Coh(\sh{A})$.

For the generic quantum Fermat quintic threefold $(X,\sh{A})$,
\[
X=\Proj\big(\mathbb{C}[x_0,\ldots,x_4]/(x_0+\ldots+x_4)\big)
\]
is a hyperplane in $\mathbb{P}^4$. Let $\{U_{ij}\}_{i\neq j}$ be the affine open cover of $X$ defined by $U_{ij}=(x_ix_j\neq 0)$.

Let $A$ be the non-commutative algebra corresponding to $\sh{A}|_{U_{01}}$. By definition, $A$ is the degree zero part of the graded algebra
\[
\left(\mathbb{C}\langle t_0,\ldots,t_4\rangle\big/\Big(\sum_k t_k^5, t_it_j-q_{ij}t_jt_i\Big)\right)\left[ \frac{1}{t_0^5},\frac{1}{t_1^5}\right].
\]
Then we can write
\begin{equation}\label{eq:local-algebra}
    A=\mathbb{C}\Big\langle u_1,\frac{1}{u_1},u_2,u_3,u_4\Big\rangle\Big/\left(1+\sum_{k=1}^4 u_k^5, u_iu_j-\overline{q}_{ij}u_ju_i\right),
\end{equation}
where $u_i=(t_it_0^4)/(t_0^5)$ and
\begin{equation}\label{eq:local-quan}
    \big(\overline{q}_{ij}\big)=\begin{pmatrix}
    1 & q^3 & q^4 & q^3 \\
    q^2 & 1 & q^4 & q^4 \\
    q & q & 1 & q^3 \\
    q^2 & q & q^2 & 1
    \end{pmatrix}.
\end{equation}

\begin{lemma}
For each $i\neq j$, there is an isomorphism $f:U_{01}\to U_{ij}$ such that $f^*(\sh{A}|_{U_{ij}})$ is isomorphic to $A$, up to a possible change of primitive root $q\in\mu_5$.
\end{lemma}

\begin{proof}
This is proved by an explicit computation. For each $i\neq j$, let $\sigma\in S_5$ be a permutation mapping $\{i,j\}$ to $\{0,1\}$. Then $\sigma$ defines a change of variables and induces an automorphism $f:U_{01}\to U_{ij}$. In this case we can compute the non-commutative algebra $f^*(\sh{A}|_{U_{ij}})$ as above. We see that there exists a permutation $\sigma$ so that $f^*(\sh{A}|_{U_{ij}})$ is equal to $A$, after a possible change of primitive root $q\in\mu_5$.
\end{proof}

Therefore it is sufficient to study finite-dimensional $A$-modules.

\subsection{Simple modules and the Ext-quiver}

As seen in \cite{KS08}, semistable objects in a Calabi--Yau-3 category should be locally given by representations of quivers with potential, and the quivers are obtained by the Ext-quivers at stable objects. See also \cite{Toda18} for the case of the category of coherent sheaves on a Calabi--Yau threefold.

For the category $\Coh(\sh{A})_{\fd}$, a finite-dimensional $\sh{A}$-module is always semistable, and it is stable if and only if it is simple. We consider the natural forgetful map
\[
\Hilb^n(\sh{A})\to \msp{M}^n(\sh{A}):=\msp{M}^{\ss,n}(\sh{A}).
\]
This is the analogue of Hilbert--Chow map. The closed points of the coarse moduli scheme $\msp{M}^n(\sh{A})$ correspond to polystable $\sh{A}$-modules, that is, semisimple $\sh{A}$-modules.

As shown in the previous section, we only need to consider finite-dimensional $A$-modules.

\begin{lemma}
All simple $A$-modules have dimension $1$ or a multiple of $5$. Furthermore, there are exactly $5$ one-dimensional $A$-module given by
\[
(u_1,u_2,u_3,u_4)=(\xi,0,0,0),
\]
where $\xi\in\mathbb{C}$ and $1+\xi^5=0$.
\end{lemma}

\begin{proof}
Let $V$ be a $d$-dimensional simple $\sh{A}$-module. We abuse the notation and write $u_i\in\End_{\mathbb{C}}(V)$ for the action of $u_i\in A$. The relations $u_i u_j=q_{ij}u_ju_i$ imply that for each $i$, $\ker(u_i)\subset V$ is an invariant subspace. Thus for each $i$, $u_i$ is either $0$ or invertible. Next, taking determinants of the relations yields
\[
\det(u_i)\det(u_j)=q_{ij}^d\det(u_j)\det(u_i).
\]
If $d$ is not a multiple of $5$, then $q_{ij}^d\neq 1$. Since $u_0$ is invertible, $\det(u_i)=0$ for all $i\neq 0$ and thus $u_i=0$. We conclude that $d=1$ and $u_0$ acts as a scalar $\xi\in\mathbb{C}$ with $1+\xi^5=0$.
\end{proof}

We observe that the $5$ one-dimensional simple $A$-modules are supported at the point $p_{01}=\pcoor{1:-1:0:0:0}\in X$, and they are all simple $A$-modules supported at $p_{01}$. We denote these $5$ simple modules by $E_i$'s, which corresponds to $\xi=-q^i$, for $i\in\mathbb{Z}/5$.

\begin{definition}
The Ext-quiver $Q$ associated to $\{E_i\}_i$ is the quiver whose vertex set $Q_0=\{E_i\}_i$ and the number of arrows from $E_i$ to $E_j$ is equal to the dimension of $\Ext^1_A\big(E_i, E_j\big)$.
\end{definition}

\begin{proposition}
The $\Ext$-quiver $Q$ associated to $(E_0,E_1,\ldots,E_4)$ is
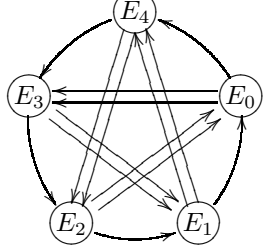
\begin{equation}\label{eq:the-quiver}
    \begin{xy} 0;<2pt,0pt>:<0pt,-2pt>:: 
    (40,16) *+[Fo]{E_0} ="0",
    (8,40) *+[Fo]{E_2} ="2",
    (32,40) *+[Fo]{E_1} ="1",
    (0,16) *+[Fo]{E_3} ="3",
    (20,0) *+[Fo]{E_4} ="4",
    (20,20) *+<2pt>{},
    "4":{\ar@<-.5ex>"2"},
    "4":{\ar@<.5ex>"2"},
    "4":{\ar@/_/"3"},
    "0":{\ar@<-.5ex>"3"},
    "0":{\ar@<.5ex>"3"},
    "0":{\ar@/_/"4"},
    "1":{\ar@<-.5ex>"4"},
    "1":{\ar@<.5ex>"4"},
    "1":{\ar@/_/"0"},
    "2":{\ar@<-.5ex>"0"},
    "2":{\ar@<.5ex>"0"},
    "2":{\ar@/_/"1"},
    "3":{\ar@<-.5ex>"1"},
    "3":{\ar@<.5ex>"1"},
    "3":{\ar@/_/"2"},
    \end{xy}
\end{equation}
\end{proposition}

\begin{proof}
The group $\Ext^1_{A}(E_i,E_j)$ is classified by extensions $0\to E_j\to F\to E_i\to 0$. The $u_i$-action on $F$ is of the form
\[
u_0=
\begin{pmatrix}
-q^j & 0 \\
0 & -q^i
\end{pmatrix}
\text{ and }
u_k=
\begin{pmatrix}
0 & a_k \\
0 & 0
\end{pmatrix}
\text{ for $k>0$.}
\]
The relations implies that $a_k(\overline{q}_{1k}q^i-q^j)=0$. Thus
\[
\dim\Ext^1_{A}(E_i,E_j)=\text{the number of $k$'s such that }\overline{q}_{1k}=q^{j-i}.
\]
\end{proof}

We denote the arrows of $Q$ by $a_i,c_i:E_i\to E_{i+3}$ and $b_i: E_i\to E_{i+4}$. From the construction of the Ext-quiver $Q$, we see that the arrows $a_i$'s, $b_i$'s, and $c_i$'s correspond to the actions of $u_2$, $u_3$, and $u_4$ respectively. Then the $q$-commuting relations translate into 
\[
\begin{aligned}[t]
   a_{i}b_{i+1}&-q^4b_{i+4}a_{i+1}&=0, \\ b_{i}c_{i+2}&-q^3c_{i+1}b_{i+2}&=0, \\
   a_{i}c_{i+2}&-q^4c_{i}a_{i+2}&=0.
\end{aligned}
\]
As expected, these relations can be patched into a potential
\[
W=\bm{b}\bm{a}\bm{c}-q^4\bm{b}\bm{c}\bm{a},
\]
where
\[
\begin{split}
    \bm{a} &=a_0+a_1+a_2+a_3+a_4, \\
    \bm{b} &=q^4b_0+b_1+qb_2+q^2b_3+q^3b_4, \\
    \bm{c} &=c_0+c_1+c_2+c_3+c_4.
\end{split}
\]
In the next section, we will show that the quiver $(Q,W)$ with potential in fact gives a local model of $\sh{A}$ near the point $p_{01}\in X$.


\section{Local models of $(X,\sh{A})$}

Let $(Q,W)$ be the quiver with potential from the previous section.

\begin{lemma}
The Jacobi algebra $\Jac(Q,W)$ is isomorphic to
\[
\mathbb{C}\langle e,u,v,w\rangle_{(\overline{q}_{ij})}\big/(e^5-1),
\]
where $(\overline{q}_{ij})$ is the quantum parameters \eqref{eq:local-quan}.
\end{lemma}

\begin{proof}
Consider the element $e=e_0+qe_1+q^2e_2+q^3e_3+q^4e_4\in\mathbb{C}Q$, where $e_i$ is the idempotent corresponding to the vertex $i$, and $u = a_0+a_1+a_2+a_3+a_4$, $v = b_0+b_1+b_2+b_3+b_4$, and $w = c_0+c_1+c_2+c_3+c_4$. It is clear that $e^5=1$ and the elements $e,u,v,w$ satisfy the appropriated $q$-commuting relations. Therefore $\mathbb{C}\langle e,u,v,w\rangle_{(\overline{q}_{ij})}\big/(e^5-1)$ is naturally a subalgebra of $\Jac(Q,W)$. 

To show that they are isomorphic, it is sufficient to show that the element $e$ generates $e_i$ for all $i$. For each $k$, we have $e^k=e_0+q^ke_1+q^{2k}e_2+q^{3k}e_3+q^{4k}e_4$. So
\[
\begin{pmatrix}
1 \\ e \\ e^2 \\ e^3 \\ e^4
\end{pmatrix}=
\begin{pmatrix}
1 & 1 & 1 & 1 & 1 \\
1 & q & q^2 & q^3 & q^4 \\
1 & q^2 & q^4 & q & q^3 \\
1 & q^3 & q & q^4 & q^2 \\
1 & q^4 & q^3 & q^2 & q
\end{pmatrix}\begin{pmatrix}
e_0 \\ e_1 \\ e_2 \\ e_3 \\ e_4
\end{pmatrix}.
\]
The matrix in the middle is a Vandermonde matrix, which is invertible. This completes the proof.
\end{proof}

\begin{remark}
One alternative description of the Jacobi algebra $\Jac(Q,W)$ is as follows. Consider the algebra
\[
\mathbb{C}\langle u,v,w\rangle_q:=\mathbb{C}\langle u,v,w\rangle\big/\left(uv-q^4vu, vw-q^4wv, wu-q^4wu \right),
\]
which is the Jacobi algebra of a quantized affine $3$-space (\cite{CMPS}). Let $G=\mu_5$ with an action on $\mathbb{C}\langle u,v,w\rangle_q$ defined by
\[
q\cdot(u,v,w)=(q^3u,q^4v,q^3w).
\]
Then the Jacobi algebra $\Jac(Q,W)$ is isomorphic to the crossed product $\mathbb{C}\langle u,v,w\rangle_q\rtimes\mu_5$, note that here the $v$ corresponds to $q^4b_0+b_1+qb_2+q^2b_3+q^3b_4$. From this perspective, the Jacobi algebra $\Jac(Q,W)$ is a quantization of the orbifold $[\mathbb{C}^3/\mu_5]$. 
\end{remark}

The Jacobi algebra $\Jac(Q,W)$ contains a subalgebra
\[
\mathbb{C}[x,y,z]:=\mathbb{C}[u^5,v^5,w^5]\subset Z(\Jac(Q,W))
\]
and is a finite $\mathbb{C}[x,y,z]$-module. Thus $\Jac(Q,W)$ can be regarded as a sheaf $\sh{J}$ of non-commutative algebras on $\mathbb{C}^3=\Spec\,\mathbb{C}[x,y,z]$. There is a canonical embedding
\begin{equation}\label{eq:can-chart}
    U_{01}\hookrightarrow\mathbb{C}^3, \left(\frac{x_2}{x_0},\frac{x_3}{x_0}, \frac{x_4}{x_0}\right)=(x,y,z) 
\end{equation}
which maps the special point $p_{01}$ to the origin. For any subset $U\subset U_{01}$, we will simply identify it with its image in $\mathbb{C}^3$ without further comment.

\begin{theorem}
For any point $p\in U_{01}$, there is an \emph{analytic} open neighborhood $U\subset U_{01}$ of $p$ such that there is a (non-unique) isomorphism
\[
\sh{A}|_U\cong\sh{J}|_U
\]
of sheaves of non-commutative algebras on $U$.
\end{theorem}

\begin{proof}
Both sheaves $\sh{A}$ and $\sh{J}$ are locally free of rank $625$, and we can write them down explicitly,
\[
\sh{A}|_U=\sh{O}_U\langle u_1,u_2,u_3,u_4\rangle_{(\overline{q}_{ij})}\big/\left( 1+\sum_{k=1}^4 u_k^5, u_2^5-x, u_3^5-y, u_4^5-z\right)
\]
and
\[
\sh{J}|_U=\sh{O}_U\langle e,u,v,w\rangle_{(\overline{q}_{ij})}\big/\left(e^5-1, u^5-x, v^5-y, w^5-z\right),
\]
where $x,y,z\in H^0(U,\sh{O}_U)$ are (holomorphic) functions corresponding to the coordinates of $U\subset\mathbb{C}^3$.

Suppose $U\subset\mathbb{C}^3$ is an analytic open subset such that $5$-th roots of the holomorphic function $-1-x-y-z$ are well-defined, that is, there exists an element
\[
f(x,y,z)\in H^0(U,\sh{O}_U)\text{ such that }f(x,y,z)^5=-1-x-y-z.
\]
Then we define a morphism of sheaves of non-commutative algebras
\[
\sh{A}|_U\to\sh{J}|_U, (u_1,u_2,u_3,u_4)\mapsto \big(f(x,y,z)e, u, v, w\big).
\]
This defines an isomorphism since $f(x,y,z)$ is non-vanishing on $U$ and thus is invertible.

Finally, $U_{01}$ is the open subset of $\mathbb{C}^3$ defined by $x+y+z\neq -1$. Therefore it can be covered by analytic open subsets satisfying the property above.
\end{proof}

Next, we analyze the Jacobi algebra $\Jac(Q,W)$ in more details.

\begin{proposition}\label{prop:local-of-J}
Let $p=(x_0,x_1,x_2)\in\mathbb{C}^3$ with $x_0\neq 0$. Then there is an analytic open neighborhood $U\subset\mathbb{C}^3$ of $p$ such that
\[
\sh{J}|_U\cong M_{5}(\mathbb{C})\tens{\mathbb{C}}\Big(\sh{O}_U[v,w]\big/\left(v^5-y, w^5-x\right)\Big),
\]
where $M_5(\mathbb{C})$ is the ring of $5$-by-$5$ matrices. Similar results also hold for points with $y_0\neq 0$ or $z_0\neq 0$.
\end{proposition}

\begin{proof}
Recall that 
\[
\sh{J}|_U=\sh{O}_U\langle e,u,v,w\rangle_{(\overline{q}_{ij})}\big/\left(e^5-1, u^5-x, v^5-y, w^5-z\right),
\]
Since $x_0\neq 0$, we can choose an analytic open neighborhood $U$ of $p$ such that there exists a holomorphic function $f(x,y,z)$ on $U$ such that $f(x,y,z)^5=x$. We consider a change of coordinates
\[
e_1=e,\quad e_2=\frac{u}{f},\quad \overline{v}=(e_1^3e_2^2)v,\quad \overline{w}=(e_1^4e_2^3)w.
\]
Then $\sh{J}|_{U}$ is generated by $e_1,e_2,\overline{v},\overline{w}$ since $e_2^5=1$. A straightforward computation shows that the elements $\overline{v},\overline{w}$ are, in fact, lying in the center $Z(\sh{J}|_U)$.

Consequently, we can write
\[
\begin{split}
   \sh{J}|_U &=\sh{O}_U\langle e_1,e_2\rangle_{(q_{12})}[\overline{v},\overline{w}]\big/\left(e_1^5-1, e_2^5-1, \overline{v}^5-y, \overline{w}^5-z \right) \\
   &\cong\Big(\mathbb{C}\langle e_1,e_2\rangle_{(q_{12})}\big/(e_1^5-1,e_2^5-1) \Big)\tens{\mathbb{C}}\Big(\sh{O}_U[\overline{v},\overline{w}]\big/(\overline{v}^5-y, \overline{w}^5-z)\Big).
\end{split}
\]
To see that the finite-dimensional Frobenius algebra 
\[
\mathbb{C}\langle e_1,e_2\rangle\big/\left(e_1e_2-q^3e_2e_1, e_1^5-1,e^2-1\right)
\]
is isomorphic to $M_5(\mathbb{C})$, one may verify that the morphism defined by
\[
e_1\mapsto\begin{pmatrix}
1 & 0 & 0 & 0 & 0 \\
0 & q & 0 & 0 & 0 \\
0 & 0 & q^2 & 0 & 0 \\
0 & 0 & 0 & q^3 & 0 \\
0 & 0 & 0 & 0 & q^4
\end{pmatrix},\quad e_2\mapsto\begin{pmatrix}
0 & 0 & 1 & 0 & 0 \\
0 & 0 & 0 & 1 & 0 \\
0 & 0 & 0 & 0 & 1 \\
1 & 0 & 0 & 0 & 0 \\
0 & 1 & 0 & 0 & 0
\end{pmatrix}
\]
is an isomorphism.
\end{proof}

\begin{remark}\label{rmk:about-J}
We may consider the finite covering $\pi:\mathbb{C}^3\to\mathbb{C}^3$, $\pi(x,y,z)=(x,y^5,z^5)$. Then for any $U\subset\mathbb{C}^3$, we have
\[
\begin{split}
    & M_5(\mathbb{C})\tens{\mathbb{C}}\Big(\sh{O}_U[v,w]\big/(v^5-y,w^5-z)\Big) \\
    \cong &\, M_5(\mathbb{C})\tens{\mathbb{C}}\pi_*\sh{O}_{\pi^{-1}(U)}=\pi_*\Big(M_5(\mathbb{C})\tens{\mathbb{C}}\sh{O}_{\pi^{-1}(U)}\Big)\\
    =&\, \pi_*\Big( M_5(\sh{O}_{\mathbb{C}^3})|_{\pi^{-1}(U)}\Big).
\end{split}
\]
This gives an equivalence between coherent $\sh{J}$-modules on $U$ and coherent $M_5(\sh{O}_{\mathbb{C}^3})$-modules on $\pi^{-1}(U)$.
\end{remark}

Now we consider a stratification
\[
\mathbb{C}^3=\mathbb{C}^3_{(0)}\coprod\mathbb{C}^3_{(1)}\coprod\mathbb{C}^3_{(2)}\coprod\mathbb{C}^3_{(3)},
\]
where $\mathbb{C}^3_{(i)}$ consists of points with exactly $i$ coordinates being non-zero. Let $p(x_0,y_0,z_0)\in\mathbb{C}^3_{(2)}$. For simplicity we assume $z_0=0$, then Proposition~\ref{prop:local-of-J} shows that there is an analytic neighborhood $U$ of $p$ such that
\[
\sh{J}|_U\cong \Big(M_{5}(\mathbb{C})^{\oplus 5}\Big)\tens{\mathbb{C}}\Big(\sh{O}_U[w]\big/\left(w^5-z\right)\Big).
\]
Similarly, for point $p\in \mathbb{C}^3_{(3)}$, there is an analytic neighborhood $U$ of $p$ such that
\[
\sh{J}|_U\cong \Big(M_{5}(\mathbb{C})^{\oplus 25}\Big)\tens{\mathbb{C}}\sh{O}_U.
\]

Finally, we recall that there is an open cover $\{U_{ij}\}_{i\neq j}$ such that all sheaves $\sh{A}|_{U_{ij}}$ of non-commutative algebras are (canonically) isomorphic. There is a natural stratification
\begin{equation}\label{eq:strat-X}
    X=X_{(0)}\coprod X_{(1)}\coprod X_{(2)}\coprod X_{(3)},
\end{equation}
where $X_{(i)}$ consists of points with exactly $i+2$ coordinates (of $\mathbb{P}^4$) being non-zero. It is clear that the canonical isomorphism $U_{01}\cong U_{ij}$ and the canonical embedding \eqref{eq:can-chart} preserve the strata.

\begin{definition}
Let $p\in X$. An \emph{analytic chart} $U$ of $p$ is an analytic open neighborhood $U\subset X$ of $p$ with an embedding $U\to\mathbb{C}^3$ mapping $p$ to the origin.
\end{definition}

Putting all above results together, we obtain the following theorem.

\begin{theorem}\label{thm:local-model}
We define the sheaves of non-commutative algebras on $\mathbb{C}^3$
\[
\begin{split}
    \sh{J}_{(0)} &= \sh{J}, \\
    \sh{J}_{(1)} &= M_{5}(\mathbb{C})\tens{\mathbb{C}}\Big(\sh{O}_{\mathbb{C}^3}[v,w]\big/\left(v^5-y, w^5-z\right)\Big), \\
    \sh{J}_{(2)} &= \Big(M_{5}(\mathbb{C})^{\oplus 5}\Big)\tens{\mathbb{C}}\Big(\sh{O}_{\mathbb{C}^3}[w]\big/\left(w^5-z\right)\Big), \\
    \sh{J}_{(3)} &= \Big(M_{5}(\mathbb{C})^{\oplus 25}\Big)\tens{\mathbb{C}}\sh{O}_{\mathbb{C}^3}. \\
\end{split}
\]
Then for all $i$ and any point $p\in X_{(i)}$, there is an analytic chart $U\to\mathbb{C}^3$ of $p$ such that
\[
\sh{A}|_{U}\cong\sh{J}_{(i)}|_{U}.
\]
Moreover, if $p\in X_{(1)}$, then the chart $U\to\mathbb{C}^3$ maps $U\cap X_{(1)}$ to the locus $(y=z=0)$; and if $p\in X_{(2)}$, then the chart $U\to\mathbb{C}^3$ maps $U\cap X_{(2)}$ to the locus $(z=0)$.
\end{theorem}

\begin{remark}\label{rmk:chart}
From the construction of the chart $U\to\mathbb{C}^3$, we see that if $p\in X_{(1)}$, then the chart maps $U\cap X_{(1)}$ to the locus $(y=z=0)$. Similarly, if $p\in X_{(2)}$, then the chart maps $U\cap X_{(2)}$ to the locus $(z=0)$. Besides, since the charts are constructed to make $5$-th roots of certain holomorphic function well-defined, we can choose finitely many charts to cover $X_{(i)}$'s for each $i$.
\end{remark}

In particular, if $p\in X_{(i)}$, then the category $\Coh(\sh{A})_p$ of coherent $\sh{A}$-modules supported at $p$ is equivalent to the category of $\sh{J}_{(i)}$-modules supported at the origin.

\begin{corollary}
Let $p\in X_{(i)}$, $i\neq 0$. There are $5^{i-1}$ simple $\sh{A}$-modules supported at $p$ and all of them are of dimension $5$. Furthermore, the Ext-quiver associated to these simple modules is the quiver consisting of $5^{i-1}$ vertices, and three loops at each vertex.
\end{corollary}

\begin{remark}
For $i\neq 0$, the sheaves $\sh{J}_{(i)}$ of algebras on $\mathbb{C}^3$ are given by the non-commutative $\mathbb{C}[x,y,z]$-algebras
\[
M_5(\mathbb{C}[x,y,z])^{\oplus 5^{i-1}},
\]
which are Morita equivalent to the commutative algebras $\mathbb{C}[x,y,z]^{\oplus 5^{i-1}}$. Moreover, these algebras are Jacobi algebras of a quiver with potential. Therefore, one may interpret Theorem~\ref{thm:local-model} as an explicit (analytic) local model of $(X,\sh{A})$ by quivers with potential.
\end{remark}

\section{Computation of Donaldson--Thomas invariants}

In this section, we will compute the generating function
\[
Z^{\sh{A}}(t)=\sum_{n=0}^{\infty}\chi_{\vir}\big(\Hilb^n(\sh{A})\big)\,t^n
\]
of degree zero DT invariants of the generic quantum Fermat quintic threefold.

\subsection{Analytic moduli spaces}

For technical reasons, we consider the category $\ConSch$ of \emph{analytic} schemes carrying an \emph{algebraic} constructible function. Objects in $\ConSch$ are pairs $(U,\nu)$ such that $U$ is an analytic open subset of an algebraic scheme $X$ and $\nu$ is a function $U\to\mathbb{Z}$ that extends to an algebraic constructible function $\overline{\nu}:X\to\mathbb{Z}$.  Morphisms from $(U_1,\nu_1)$ to $(U_2,\nu_2)$ are analytic morphisms $f:U_1\to U_2$ such that $\nu_2\circ f=\nu_1$. Also, the product exists in the category $\ConSch$ given by
\[
(U_1,\nu_1)\times (U_2,\nu_2)=(U_1\times U_2, \nu_1\times\nu_2),
\]
where $(\nu_1\times\nu_2)(x_1,x_2)=\nu_1(x_1)\nu_2(x_2)$.

It is clear that if $(U_1,\nu_1)$ and $(U_2,\nu_2)$ are isomorphic in $\ConSch$, then $\chi(U_1,\nu_1)=\chi(U_2,\nu_2)$.

\begin{lemma}
Let $X$ and $Y$ be schemes. If $X$ and $Y$ are analytic local isomorphic, that is, there exist analytic open covers $\{U_{\alpha}\}_{\alpha}$ of $X$ and $\{V_{\alpha}\}_{\alpha}$ of $Y$ such that for each $\alpha$, there is an analytic isomorphism $f_{\alpha}:U_{\alpha}\to V_{\alpha}$. Then
\[
\chi_{\vir}(X)=\chi_{\vir}(Y).
\]
\end{lemma}

\begin{proof}
Since the Behrend function depends only on the analytic topology of a scheme, the analytic isomorphism $f_{\alpha}$ induces an isomorphism
\[
f_{\alpha}:(U_{\alpha},\nu_{U_{\alpha}})\to (V_{\alpha},\nu_{V_{\alpha}})
\]
in $\ConSch$. Then the result follows from the fact that Euler characteristics can be computed from an open cover. 
\end{proof}


\begin{definition}
Let $f:X\to Y$ be a morphism of schemes, $F$ a scheme with constructible functions $\mu:X\to\mathbb{Z}$, $\nu:F\to\mathbb{Z}$. We say
\[
f:(X,\mu)\to Y
\]
is an \emph{analytic local} fibration with fibre $(F,\nu)$ if there is an analytic open cover $\{U_{\alpha}\}_{\alpha}$ of $Y$ such that for each $\alpha$, there is an isomorphism
\[
\big(f^{-1}(U_{\alpha}),\mu\big)\cong (U_{\alpha},1)\times(F,\nu)
\]
in $\ConSch$. Note that $f:(X,\mu)\to (Y,1)$ is generally not a morphism in $\ConSch$.
\end{definition}

\begin{lemma}
If $f:(X,\mu)\to Y$ is an analytic local fibration with fibre $(F,\nu)$, then
\[
\chi(X,\mu)=\chi(Y)\cdot\chi(F,\nu).
\]
\end{lemma}

\begin{proof}
It is easy to see that for any $c\in\mathbb{Z}$, $f:\mu^{-1}(c)\to Y$ is an analytic-local fibration with fibre $\nu^{-1}(c)$, thus
\[
\chi(\mu^{-1}(c))=\chi(Y)\cdot\chi(\nu^{-1}(c)).
\]
\end{proof}
 
Since our local models of $(X,\sh{A})$ are given in analytic topology, it is not obvious that it gives an analytic isomorphism between algebraic moduli spaces.

For the rest of the section, let $X$ be a quasi-projective smooth variety and $\sh{A}$ a locally free sheaf of non-commutative algebras on $X$.

\begin{definition}
The Hilbert--Chow map is the composition
\[
\Hilb^n(\sh{A})\to \msp{M}^n(\sh{A}) \to \msp{M}^n(X)\cong\Sym^n(X),
\]
where the middle morphism sends a finite-dimension $\sh{A}$-module to its underlying coherent sheaves on $X$, which has zero-dimensional support and is of length $n$.
\end{definition}

For any analytic or algebraic subset $S\subset X$, we define the fiber product
\[
\xymatrix{
\Hilb^n(\sh{A})_{S}\ar@{^{(}->}[r]\ar[d]\ar@{}[dr]|{\Box} & \Hilb^n(\sh{A})\ar[d] \\
\Sym^n(S)\ar@{^{(}->}[r] & \Sym^n(X)
}
\]
in the appropriate category. In other words, $\Hilb^n(X,\sh{A})_S$ parameterizes $\sh{A}$-module quotients supported in $S$. Clearly if $S$ is analytic open in $X$, then $\Hilb^n(\sh{A})$ is also analytic open in $\Hilb^n(\sh{A})$. Also if $S$ is a locally closed subscheme of $X$, then $\Hilb^n(\sh{A})$ is also a locally closed subscheme in $\Hilb^n(\sh{A})$

We first prove the equivalence between algebraic and analytic finite-dimensional $\sh{A}$-modules.

\begin{lemma}\label{lem:GAGA}
The analytification defines an equivalence
\[
\Coh(\sh{A})_{\fd}\cong\Coh(\sh{A}^{\mathrm{an}})_{\fd}
\]
of categories, where $\Coh(\sh{A}^{\mathrm{an}})_{\fd}$ is the category of analytic coherent sheaves on $X$ with an $\sh{A}^{\mathrm{an}}$-action. 
\end{lemma}

\begin{proof}
As mentioned before, this statement is Zariski local on $X$. We may assume $X=\Spec(R)$ is affine, and then $\sh{A}$ is a non-commutative $R$-algebra.

First, we may choose a compactification $X\subset\overline{X}$ so we can apply GAGA theorem \cite{GAGA}. Since analytification preserves supports of coherent sheaves, it induces an equivalence between $\Coh(\sh{O}_X)_{\fd}$ and $\Coh(\sh{O}_X^{\mathrm{an}})_{\fd}$.

While $\sh{O}_X^{\mathrm{an}}$ and $\sh{A}^{\mathrm{an}}$ are not in the category $\Coh(\sh{O}_X^{\mathrm{an}})_{\fd}$, for any $\sh{F}^{\mathrm{an}}$ in $\Coh(\sh{O}_X^{\mathrm{an}})_{\fd}$, the ring $\Hom_{\sh{O}_X^{\mathrm{an}}}(\sh{F}^{\mathrm{an}},\sh{F}^{\mathrm{an}})$ is naturally an $R$-algebra. By definition, an $\sh{A}^{\mathrm{an}}$-module is given by an analytic coherent sheaf $\sh{F}^{\mathrm{an}}$ with a morphism
\[
A\to \Hom_{\sh{O}_X^{\mathrm{an}}}(\sh{F}^{\mathrm{an}},\sh{F}^{\mathrm{an}})
\]
of $R$-algebras, which must be algebraic. Therefore the analytification
\[
\Coh(\sh{A})_{\fd}\cong\Coh(\sh{A}^{\mathrm{an}})_{\fd}
\]
is an equivalence of categories.
\end{proof}

\begin{proposition}
Let $U$ be an analytic open subset of $X$. Then analytification of $\Hilb^n(\sh{A})_U$ is the analytic moduli space parameterizing analytic cyclic $\sh{A}^{\mathrm{an}}|_{U}$-modules.
\end{proposition}

\begin{proof}
We first note that such analytic moduli space exists. Since $\sh{A}^{\mathrm{an}}$ is locally free, there is a Quot space  $\mathcal{Q}$ parameterizing quotients of $\sh{A}^{\mathrm{an}}$, and we take $\mathcal{M}$ to the closed subspace of $\mathcal{Q}$ consisting of points $[\sh{A}^{\mathrm{an}}\to\sh{F}^{\mathrm{an}}]$ such that the kernel is $\sh{A}^{\mathrm{an}}$-invariant.

Since any family of algebraic $\sh{A}$-modules is analytic, there is a canonical morphism
\begin{equation}\label{eq:temp3}
    \Hilb^n(\sh{A})_U\to\mathcal{M},
\end{equation}
which is bijective from the previous lemma. Note that the previous lemma also works for any family $\sh{A}$-modules over a proper scheme (where GAGA applies). It gives an equivalence between infinitesimal deformations of algebraic and analytic $\sh{A}$-modules. In other words, \eqref{eq:temp3} is \'{e}tale and hence is an analytic isomorphism.
\end{proof}

\begin{corollary}\label{cor:GAGA}
Let $\sh{A}_1$ and $\sh{A}_2$ be two coherent sheaves of non-commutative algebras on $X$. Suppose there is an analytic open subset $U\subset X$ such that $\sh{A}_1|_{U}\cong\sh{A}_2|_{U}$. Then there is an analytic isomorphism
\[
\Hilb^n(\sh{A}_1)_{U}\cong\Hilb^n(\sh{A}_2)_{U}.
\]
\end{corollary}

\subsection{Stratifying Hilbert schemes of points}

For any locally closed subset $Z\subset X$, we have $\Hilb^n(\sh{A})_Z$, the locally closed subscheme of $\Hilb^n(\sh{A})$ parameterizing $\sh{A}$-module quotients supported in $Z$ with the induced Hilbert--Chow map $\Hilb^n(\sh{A})_Z\to\Sym^n(Z)$. Let
\[
Z^{\sh{A}}_Z(t)=\sum_{n=0}^{\infty}\,\chi_{\vir}\big(\Hilb^n(\sh{A}),\Hilb^n(\sh{A})_Z\big)\,t^n
\]
be the generating function.

\begin{proposition}\label{prop:hilb_strata_sup}
Let $X=X_1\coprod X_2\coprod\ldots\coprod X_r$ be a stratification of $X$. Then
\[
Z^{\sh{A}}(t)=\prod_{i=1}^r Z^{\sh{A}}_{X_{i}}(t).
\]
\end{proposition}

\begin{proof}
Write $\pi_n$ for the set of $r$-tuples $(n_1,\ldots,n_r)$ of non-negative integers such that $n_1+\ldots+n_r=n$. The stratification $X=\coprod_i X_{i}$ induces a stratification
\[
\Hilb^n(\sh{A})=\coprod_{(n_1,\ldots,n_r)\in \pi_n}\Hilb^n(\sh{A})_{(n_1,\ldots,n_r)},
\]
where $\Hilb^n(\sh{A})_{(n_1,\ldots,n_r)}$ parameterizes quotients $\sh{A}\to\sh{F}$ such that $\sh{F}|_{X_i}$ is of length $n_i$ for all $i$. Then
\begin{equation}\label{eq:hilb_strata_sup}
    \Hilb^n(\sh{A})_{(n_1,\ldots,n_r)}\cong\prod_{i=1}^r\Hilb^{n_i}(\sh{A})_{X_i}.
\end{equation}

It remains to show that the isomorphism \eqref{eq:hilb_strata_sup} induces an isomorphism
\[
\big(\Hilb^n(\sh{A})_{(n_1,\ldots,n_r)},\nu_n\big)\cong\prod_{i=1}^r\big(\Hilb^{n_i}(\sh{A})_{X_i},\nu_{n_i}\big)
\]
in $\ConSch$, where $\nu_n$ is the Behrend function on $\Hilb^n(\sh{A})$. Let $p=[\sh{A}\to\sh{F}]\in\Hilb^n(\sh{A})$ be a closed point. We choose analytic open subsets $U_i\subset X$ such that $\overline{U_i}\cap\overline{ U_j}=\varnothing$ and 
\[
\supp(\sh{F})\cap X_i\subset U_i.
\]
Then $\Hilb^n(\sh{A})_{\coprod_i U_i}$ is analytic open in $\Hilb^n(\sh{A})_{\coprod_i U_i}\subset\Hilb^n(\sh{A})$, and 
\[
\Hilb^n(\sh{A})_{\coprod_i U_i}\cong\prod_i\Hilb^{n_i}(\sh{A})_{U_i}\subset\prod_i\Hilb^{n_i}(\sh{A}).
\]
Thus
\[
\begin{split}
    \big(\Hilb^n(\sh{A})_{\coprod_i U_i},\nu_n\big) &=\big(\Hilb^n(\sh{A})_{\coprod_i U_i},\nu_{\Hilb^n(\sh{A})_{\coprod_i U_i}}\big) \\
    & \cong \prod_i\big(\Hilb^{n_i}(\sh{A})_{U_i},\nu_{\Hilb^{n_i}(\sh{A})_{U_i}}\big)=\prod_i\big(\Hilb^{n_i}(\sh{A})_{U_i},\nu_{n_i}\big).
\end{split}
\]
which completes the proof.
\end{proof}

As in the classical case, $\Hilb^n(X,\sh{A})_S$ has a standard stratification indexed by partitions of $n$. The Hilbert--Chow map sends the stratum $\Hilb^n(X,\sh{A})_{S,(n)}$ to the diagonal $S\hookrightarrow\Sym^n(S)$.

\begin{proposition}\label{prop:local-generating}
Suppose the induced morphism
\[
\big(\Hilb^n(\sh{A})_{S,(n)},\nu_n\big)\to S
\]
is an analytic-local fibration with fibre $(F_n,\mu_{n})$. Then
\[
Z^{\sh{A}}_S(t)= \left(\sum_{n=0}^{\infty}\,\chi(F_n,\mu_n)\,t^n\right)^{\chi(S)}.
\]
\end{proposition}

\begin{proof}
This is the analogue of \cite[Theorem 4.11]{BF08}, but we use analytic open cover instead of \'{e}tale cover. The main idea is that there is a stratification
\[
\Hilb^n(\sh{A})_{S}=\coprod_{\alpha\vdash n}\,\Hilb^n(\sh{A})_{S,\alpha}
\]
and for each $\alpha=(\alpha_1\leq\ldots\leq\alpha_r)\vdash n$,
\[
\Hilb^n(\sh{A})_{S,\alpha}\subset\prod_{i=1}^r\Hilb^{\alpha_i}(\sh{A})_{S,(\alpha_i)}.
\]
By the same argument in Proposition~\ref{prop:hilb_strata_sup}, we get that
\[
\big(\Hilb^n(\sh{A})_{S,\alpha},\nu_n\big)\subset\prod_{i=1}^r\big(\Hilb^{\alpha_i}(\sh{A})_{S,(\alpha_i)},\nu_{\alpha_i}\big)
\]
is an immersion in $\ConSch$, and by assumption, the right hand side is an analytic-local fibration over $S$ with fibre $(F_{\alpha_i},\nu_{\alpha_i})$. Then the formula follows from a standard calculation.
\end{proof}

\subsection{Computation of DT invariants}

Let $(X,\sh{A})$ be the quantum Fermat quintic threefold. We apply Proposition~\ref{prop:hilb_strata_sup} to the stratification \eqref{eq:strat-X} of $X$ and obtain
\[
Z^{\sh{A}}(t)=\prod_{i=0}^3\, Z^{\sh{A}}_{X_{(i)}}(t).
\]

\begin{theorem}\label{thm:fibration}
For each $i$, 
\[
(\Hilb^n(\sh{A})_{X_{(i)},(n)},\nu_n)\to X_{(i)}
\]
is an analytic local fibration with fibre $(\Hilb^n\big(\mathbb{C}^3,\sh{J}_{(i)})_0,\mu_n\big)$, where $\sh{J}_{(i)}$'s are the sheaves of algebras on $\mathbb{C}^3$ defined in Theorem~\ref{thm:local-model}, and $\mu_n$ is the Behrend function of $\Hilb^n(\mathbb{C}^3,\sh{J}_{(i)})$.
\end{theorem}

\begin{proof}
We only prove for the case $i=1$, as the proofs of the other cases are similar. First, we choose an analytic open cover $\{U_{\alpha}\}$ of charts. Since $\sh{A}|_{U_{\alpha}}\cong\sh{J}_{(1)}|_{U_{\alpha}}$, we have an analytic isomorphism
\[
\Hilb^n(\sh{A})_{U_{\alpha}}\cong\Hilb^n(\mathbb{C}^3,\sh{J}_{(1)})_{U_{\alpha}},
\]
by Corollary~\ref{cor:GAGA}, which gives an isomorphism
\begin{equation}\label{eq:temp1}
    \big(\Hilb^n(\sh{A})_{U_{\alpha}},\nu_n\big)\cong\big(\Hilb^n(\mathbb{C}^3,\sh{J}_{(1)})_{U_{\alpha}},\mu_n\big)
\end{equation}
in $\ConSch$. Let $Z$ be the locus of $\mathbb{C}^3$ defined by $y=z=0$. Observe that
\[
\sh{J}_{(1)} = M_{5}(\mathbb{C})\tens{\mathbb{C}}\Big(\sh{O}_{\mathbb{C}^3}[v,w]\big/\left(v^5-y, w^5-z\right)\Big)
\]
is invariance under the translation of $x$, the Hilbert--Chow map
\[
\big(\Hilb^n(\mathbb{C}^3,\sh{J}_{(1)})_{Z,(n)},\mu_n)\to Z
\]
is a \emph{Zariski} local fibration with fibre $\big(\Hilb^n(\mathbb{C}^3,\sh{J}_{(i)})_0,\mu_n\big)$. Here we use the fact that the Behrend function is constant on orbits of a group action.

Finally, since the chart $U_{\alpha}\to\mathbb{C}^3$ map $X_{(1)}$ into $Z$ (see Remark~\ref{rmk:chart}), the isomorphism \eqref{eq:temp1} induces an isomorphism 
\[
\big(\Hilb^n(\sh{A})_{U_{\alpha}\cap X_{(1)}},\nu_n\big)\cong\big(\Hilb^n(\mathbb{C}^3,\sh{J}_{(1)})_{U_{\alpha}\cap Z},\mu_n\big)
\]
Then the theorem follows.
\end{proof}

Next we deal with the sheaves $\sh{J}_{(i)}$ of algebras before we state our main theorem. Recall that for $i\neq 0$, $\sh{J}_{(i)}$ can be written as direct sums of $M_5(\sh{O}_{\mathbb{C}^3})$ (see Remark~\ref{rmk:about-J}).

\begin{lemma}
If $n$ is not a multiple of $5$, then $\Hilb^n\big(\mathbb{C}^3,M_5(\sh{O}_{\mathbb{C}^3})\big)=\varnothing$. For $n=5k$, there is a canonical isomorphism
\[
\Hilb^{5k}\big(\mathbb{C}^3, M_5(\sh{O}_{\mathbb{C}^3})\big)\cong\Quot^k(\mathbb{C}^3,\sh{O}_{\mathbb{C}^3}^{\oplus 5}).
\]
\end{lemma}

\begin{proof}
The Morita equivalence 
\[
\Coh(\mathbb{C}^3)\to \Coh(M_5(\sh{O}_{\mathbb{C}^3}))
\]
is given by $\mathbb{C}^{\oplus 5}\otimes_{\mathbb{C}}-$, where $\mathbb{C}^{\oplus 5}$ is the canonical representation of $M_5(\mathbb{C})$. This implies that the dimension of a $M_5(\sh{O}_{\mathbb{C}^3})$-module must be a multiple of $5$.

The isomorphism is a direct consequence of the following fact: Let $R$ be a commutative ring, and $M$ be an $R$-module. For any $n$, we consider the simple $M_n(R)$-module $R^{\oplus n}$. Then a morphism
\[
s:M_n(R)\to R^{\oplus n}\otimes M
\]
of $M_n(R)$-modules is surjective if and only if the induced morphism
\[
\delta(s): R^{\oplus n}\hookrightarrow M_n(R)\xrightarrow{s} R^{\oplus n}\otimes M\to M
\]
of $R$-modules is surjective, where the first map is the diagonal map, and the last map is defined by $(r_1,\ldots,r_n)\otimes m\mapsto\sum_i r_im$.
\end{proof}

\begin{theorem}\label{thm:main-computation}
We have
\[
Z^{\sh{A}}(t)=\Big(Z^{Q,W}(t)\Big)^{10}\Big(M(-t^5)\Big)^{-50},
\]
where $M(t)$ is the MacMahon function.
\end{theorem}

\begin{proof}
We use Proposition~\ref{prop:local-generating} with Theorem~\ref{thm:fibration} to obtain
\[
Z^{\sh{A}}(t)=\prod_{i=0}^3 \left(Z^{\mathbb{C}^3,\sh{J}_{(i)}}_0(t)\right)^{\chi(X_{(i)})}.
\]
For $i\neq 0$, we have
\[
\begin{split}
    Z^{\mathbb{C}^3,\sh{J}_{(i)}}_0(t)=&\; Z^{\mathbb{C}^3,M_5(\sh{O})^{\oplus 5^{i-1}}}_0(t) \\
    = &\, \sum_{k=0}^{\infty} \,\chi_{\vir}\Big(\Quot^k\big((\mathbb{C}^3)^{\coprod 5^{i-1}},\sh{O}^{\oplus 5}\big),\Quot^k\big((\mathbb{C}^3)^{\coprod 5^{i-1}},\sh{O}^{\oplus 5}\big)_0\Big)\, t^{5k} \\
    = &\, \left(\sum_{k=0}^{\infty} \,\chi_{\vir}\big(\Quot^k(\mathbb{C}^3,\sh{O}^{\oplus 5}),\Quot^k(\mathbb{C}^3,\sh{O}^{\oplus 5})_0\big)\, t^{5k}\right)^{5^{i-1}}.
\end{split}
\]
For the first equality, see Remark~\ref{rmk:about-J}. The second equality is given by the previous lemma, and the third equality is a standard result about Hilbert schemes of a disjoint union of schemes.

Now, DT invariants of Quot schemes of points on $\mathbb{C}^3$ are well-known:
\[
\sum_{k=0}^{\infty} \,\chi_{\vir}\big(\Quot^k(\mathbb{C}^3,\sh{O}^{\oplus 5}),\Quot^k(\mathbb{C}^3,\sh{O}^{\oplus 5})_0\big)\, t^{k}= M(-t)^5.
\]
We conclude that
\[
Z^{\sh{A}}(t)=Z_0^{\mathbb{C}^3,\sh{J}}(t)^{\chi(X_{(0)})}\cdot \Big(M(-t^5)^5\Big)^{\chi(X_{(1)})+5\chi(X_{(2)})+25\chi(X_{(3)})},
\]
with $\chi(X_{(0)})=0$ and $\chi(X_{(1)})+5\chi(X_{(2)})+25\chi(X_{(3)})=-10$.

Finally, recall that $\sh{J}_{(i)}$'s also are local models of $\sh{J}$ on strata $\mathbb{C}^3_{(i)}$ of $\mathbb{C}^3$. We repeat all above arguments to $(\mathbb{C}^3,\sh{J})$ which lead to
\[
\begin{split}
    Z^{\mathbb{C}^3,\sh{J}}(t) &=Z_0^{\mathbb{C}^3,\sh{J}}(t)^{\chi(\mathbb{C}^3_{(0)})}\cdot \Big(M(-t^5)^5\Big)^{\chi(\mathbb{C}^3_{(1)})+5\chi(\mathbb{C}^3_{(2)})+25\chi(\mathbb{C}^3_{(3)})} \\
    &=Z_0^{\mathbb{C}^3,\sh{J}}(t).
\end{split}
\]
since $\chi(\mathbb{C}^3_{(3)})=\chi(\mathbb{C}^3_{(3)})=\chi(\mathbb{C}^3_{(3)})=0$. The sheaf $\sh{J}$ on $\mathbb{C}^3$ is given by the Jacobi algebra $\Jac(Q,W)$, and by definition, finite-dimensional quotients of $\sh{J}$ are exactly framed representations of $(Q,W)$ with framing vector $(1,1,1,1,1)$. Hence by our definition, $Z^{\mathbb{C}^3,\sh{J}}(t)=Z^{Q,W}(t)$.
\end{proof}

For the generating function $Z^{Q,W}$, we have seen that $(Q,W)$ can be viewed as a quantization of an orbifold $[\mathbb{C}^3/\mu_5]$, and DT invariants of an orbifold is known to be related to \emph{colored plane partitions} \cite{You10}. We will discuss more of this in the next section, and show that $Z^{Q,W}$ can be computed using combinatorics.



\subsection{A geometric interpertation}

As a final remark, we give a possible geometric interpretation of this result. Recall that finite-dimension $\sh{A}$-modules are of dimension $1$ or $5$. We denote $M^{d}_{\simple}$ the moduli space of $d$-dimensional $\sh{A}$-modules. Then Theorem~\ref{thm:local-model} implies that
\begin{enumerate}
    \item[(a)] There is a morphism $M^{1}_{\simple}\to X_{(0)}$ which is a $\mu_5$-torsor.
    \item[(b)] $M^{5}_{\simple}$ is smooth, and there is a (ramified) covering $M^{5}_{\simple}\to X\setminus X_{(0)}$, which is $5^{i-1}$-to-$1$ on the stratum $X_{(i)}$.
\end{enumerate}
In particular,
\[
\chi(M^{5}_{\simple})=\chi(X_{(1)})+5\chi(X_{(2)})+25\chi(X_{(3)}).
\]
The factor $M(-t^5)^{-50}$ can be expressed as
\[
Z^{M^{5}_{\simple}, M_5(\mathbb{\sh{O}})}(t)=\sum_{k=0}^{\infty} \chi_{\vir}\big(\Quot^{k}(M^{5}_{\simple},\sh{O}^{\oplus 5})\big)\, t^{5k}
\]
On the other hand, the $M^1_{\simple}$ consists of $50$ points, and if we consider the Ext-quiver $\tilde{Q}$ associated to $M^1_{\simple}$, then $\tilde{Q}$ is the disjoint union of $10$ copies of the quiver $Q$. Therefore we can write
\[
Z^{Q,W}(t)^{10}=Z^{\tilde{Q},\tilde{W}}(t)
\]
and
\[
Z^{\sh{A}}(t)=Z^{\tilde{Q},\tilde{W}}(t)\cdot Z^{M^{5}_{\simple}, M_5(\mathbb{\sh{O}})}(t)
\]
This strongly suggests that the DT invariants of the Calabi--Yau-3 category $\Coh(\sh{A})_{\fd}$ can be computed \emph{directly} from the moduli space of simple objects in $\Coh(\sh{A})_{\fd}$ and the Ext-quivers between simple objects.

\section{Multi-colored plane partitions}\label{sec:mcpp}

A plane partition is a finite subset $\pi$ of $\mathbb{Z}_{\geq 0}^{\oplus 3}$ such that if any of $(i+1,j,k)$, $(i,j+1,k)$, $(i,j,k+1)$ are in $\pi$, then so is $(i,j,k)$. We will refer points in a plane partition as ``boxes'', and a plane partition can be viewed as a pile of boxes stacked in the positive octant. The size $|\pi|$ is the number of boxes. We denote by $\planep$ the set of plane partitions.

\subsection{Colored plane partitions}

In this section, we recall the notion of colored plane partitions and their relation with orbifolds. These results are taken from \cite{You10}, \cite{BCY} (see also \cite{DOS}).

Let $G=\mu_r$ be the finite group of $r$-th roots of unity in $\mathbb{C}$. We consider the $\mu_r$-action on $\mathbb{C}^3$ with weights $(a,b,c)$. We will denote this action $\mu_5(a,b,c)$. We identify $\hat{G}$, the abelian group of characters, with $\mathbb{Z}/r\mathbb{Z}$.

\begin{definition}
A $\mu_r(a,b,c)$-colored plane partition is a plane partition $\pi\in\planep$ with the coloring $K:\pi\to\mathbb{Z}/r\mathbb{Z}$ defined by
\[
K(i,j,k)=ai+bj+ck.
\]
For each color $i$, let $|\pi|_i$ be the number of boxes in $\pi$ with color $i$.
\end{definition}

We define the generating function of $\mu_r(a,b,c)$-colored plane partitions
\[
Z^{\mu_r(a,b,c)}_{\PL}(t_0,\ldots,t_{r-1})=\sum_{\pi\in\planep} \,t_0^{|\pi|_0}\cdots t_{r-1}^{|\pi|_{r-1}}.
\]

For the $\mu_r(a,b,c)$-action, we consider the \emph{McKay quiver} $Q_r(a,b,c)$ whose vertices correspond to irreducible representations of $\mu_5$. Thus the set $Q_r(a,b,c)_0$ of vertices is identified with $\hat{\mu_r}\cong\mathbb{Z}/r\mathbb{Z}$. Arrows of $Q_r(a,b,c)$ are
\[
x_i:i\to i+a,\quad y_i:i\to i+b,\quad z_i:i+c
\]
for all vertex $i$. There is a natural potential
\[
W=\bm{y}\bm{x}\bm{z}-\bm{y}\bm{z}\bm{x},
\]
where $\bm{x}=\sum_i x_i$, $\bm{y}=\sum_i y_i$, and $\bm{z}=\sum_i z_i$.

Given any plane partition $\pi$, the $\mu_r(a,b,c)$-coloring defines a dimension vector
\[
|\pi|:=(|\pi|_0,\ldots,|\pi|_{r-1})\in\mathbb{Z}^{\oplus Q_r(a,b,c)_0}.
\]

On the other hand, the $\mu_r(a,b,c)$-action defines an orbifold $\orbi{X}=[\mathbb{C}^3/\mu_r]$. For any $\rho\in K_0(\text{Rep}(\mu_r))$, we consider the Hilbert scheme $\Hilb^{\rho}(\orbi{X})$ parameterizing $\mu_5$-invariant subschemes $Z\subset\mathbb{C}^3$ such that the induced $\mu_5$-representation on $\sh{O}_Z$ is in the class $\rho$. The group $K_0(\text{Rep}(\mu_r))$ is canonically identified with $\mathbb{Z}^{\oplus Q_r(a,b,c)_0}$. Furthermore, it is well-known that the Hilbert scheme $\Hilb^{\dd}(\orbi{X})$ is isomorphic to the fine moduli space $\msp{M}^{\dd,e_0}(Q_r(a,b,c),W)$ of framed representations of the quiver $Q_r(a,b,c)$ with potential $W$, with the framing vector $e_0=(1,0,\ldots,0)$. We define the generating function
\[
Z^{\orbi{X}}(t_0,\ldots,t_{r-1})=\sum_{\dd=(d_0,\ldots,d_{r-1})}\chi_{\vir}\big(\Hilb^{\dd}(\orbi{X})\big)\,t_0^{d_0}\cdots t_{r-1}^{d_{r-1}}
\]
of DT invariants of the orbifold $\orbi{X}$, which is equal to the generating function $Z^{Q_r(a,b,c),W,e_0}$ in our notation.

\begin{proposition}[\cite{BCY}]
The generating function $Z^{\orbi{X}}$ is, up to signs, given by the generating function $Z^{\mu_r(a,b,c)}_{\PL}$ of colored plane partitions. More specifically, we have
\[
Z^{\orbi{X}}(t_0,\ldots,t_{r-1})=\sum_{\pi\in\planep}(-1)^{|\pi|_0+\langle |\pi|,|\pi|\rangle}\, t_0^{|\pi|_0}\cdots t_{r-1}^{|\pi|_{r-1}},
\]
where $\langle -,-\rangle$ is the bilinear form associated to the quiver $Q_r(a,b,c)$. 
\end{proposition}

\subsection{Multi-colored plane partitions}

Let $Q=(Q_0,Q_1)$ be a quiver with a labeling of arrows $\ell:Q_1\to\{x,y,z\}$. We denote $\mathcal{S}(Q_0)$ be the set of non-empty subsets of $Q_0$.

\begin{definition}
A $Q$-multi-colored plane partition consists of a plane partition $\pi\in\planep$ with a multi-coloring 
\[
K:\pi\to\mathcal{S}(Q_0)
\]
such that for any arrow $a:v\to w$ labeled with $x$, if $w\in K(i,j,k)$ for some $(i,j,k)\in\pi$, then $v\in K(i-1,j,k)$, and similar conditions hold for arrows labeled with $y$ and $z$. Note that there are many different $Q$-multi-colorings on one plane partition $\pi$.
\end{definition}

Given a $Q$-multi-colored plane partition $\overline{\pi}=(\pi,K)$, it associates a dimension vector $\ww(\overline{\pi}):=\big(w_v(\overline{\pi})\big)_{v\in Q_0}$ defined by
\[
\ww_v(\overline{\pi})=\text{number of boxes $(i,j,k)\in\pi$ such that the color $v\in K(i,j,k)$}.
\]
For any dimension vector $\dd\in\mathbb{Z}^{\oplus Q_0}$, we denote by $n_Q(\dd)$ the number of $Q$-multi-colored plane partitions with dimension vector $\dd$. We define the generating function
\[
Z^{Q}_{\PL}(\tt)=\sum_{\dd} n_Q(\dd)\,\tt^{\dd}
\]
of $Q$-multi-colored plane partitions, where we write $\tt$ for $(t_v)_{v\in Q_0}$, $\dd=(d_v)_{v\in Q_0}$ and $\tt^{\dd}=\prod_{\dd\in Q_0} t_v^{d_v}$.

For simplicity, from now on we only consider the quiver $(Q,W)$ \eqref{eq:the-quiver} with potential from the quantum Fermat quintic threefold. We will rearrange the vertices and arrows, and write
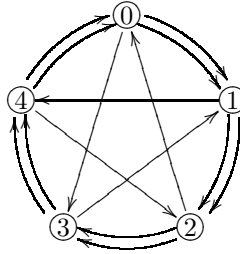
\[
\begin{xy} 0;<2pt,0pt>:<0pt,-2pt>:: 
(40,16) *+[Fo]{1} ="0",
(8,40) *+[Fo]{3} ="2",
(32,40) *+[Fo]{2} ="1",
(0,16) *+[Fo]{4} ="3",
(20,0) *+[Fo]{0} ="4",
(20,20) *+<2pt>{},
"4":{\ar@<-.5ex>@/^/"0"},
"4":{\ar@<.5ex>@/^/"0"},
"4":{\ar"2"},
"0":{\ar@<-.5ex>@/^/"1"},
"0":{\ar@<.5ex>@/^/"1"},
"0":{\ar"3"},
"1":{\ar@<-.5ex>@/^/"2"},
"1":{\ar@<.5ex>@/^/"2"},
"1":{\ar"4"},
"2":{\ar@<-.5ex>@/^/"3"},
"2":{\ar@<.5ex>@/^/"3"},
"2":{\ar"0"},
"3":{\ar@<-.5ex>@/^/"4"},
"3":{\ar@<.5ex>@/^/"4"},
"3":{\ar"1"},
\end{xy}
\]
with outer arrows $x_i,y_i:i\to i+1$ and inner arrows $z_i:i+3$. The induced potential is
\[
W=\bm{y}\bm{x}\bm{z}-q\bm{y}\bm{z}\bm{x},
\]
where $\bm{x}=\sum x_i$, $\bm{y}=\sum y_i$, and $\bm{z}=z_0+qz_1+q^2z_2+q^3z_3+q^4z_4$.

\vspace{1mm}
Our main theorem of the section is to associate DT invariants of $(Q,W)$ with $Q$-multi-colored plane partitions.

\begin{theorem}\label{thm:DT-to-part}
The generating function $Z^{Q,W}$ is, up to signs, given by the generating function $Z^{Q}_{\PL}$ of $Q$-multi-colored plane partitions. More specifically, we have
\[
Z^{Q,W}(\tt)=\sum_{\dd\in N_Q^+}(-1)^{|\dd|+\langle |\dd|,|\dd|\rangle}\, n_Q(\dd)\,\tt^{\dd},
\]
where $|\dd|=\sum_i d_i$ and $\langle -,-\rangle$ is the bilinear form associated to the quiver $Q$. 
\end{theorem}

We will leave the proof to \ref{pf:part}.

\vspace{1mm}

Observe that this quiver is the same as the McKay quiver of $\mu_5(1,1,3)$. Recall that the vertices of the McKay quiver correspond to the irreducible representations of $\mu_5$, in which there is a distinguished one, the trivial representation. However, the vertices of $Q$ are simple $\sh{A}$-modules. There are exactly $5$ ways to identify the quiver $Q$ with the McKay quiver of $\mu_5(1,1,3)$, depending on a choice of a vertex in $Q_0$. In other words, for each $v\in Q_0$, there is a unique bijection $\alpha_v:Q_0\to\mathbb{Z}/5\mathbb{Z}$ with $\alpha_v(v)=0$ identifying the quiver $Q$ with the McKay quiver of $\mu_5(1,1,3)$.

\begin{definition}
Let $v\in Q_0$ be a vertex. A $(Q,v)$-colored plane partition is a plane partition $\pi$ with the coloring $K_v:\pi\to Q_0$ making $\alpha_v\circ K_v:\pi\to\mathbb{Z}/5\mathbb{Z}$ the $\mu_5(1,1,3)$-coloring of $\pi$.
\end{definition}

Given five plane partitions $\pi_v$ indexed by $Q_0$. We may define a $Q$-multi-colored plane partition by taking $\pi$ to be the union of $\pi_v$'s with multi-coloring
\[
K(i,j,k)=\left\{K_v(i,j,k): (i,j,k)\in\pi_v\right\}\subset Q_0.
\]

\begin{lemma}
Any $Q$-multi-colored plane partition is uniquely determined by the $(Q,v)$-colored plane partitions $\pi_v$ for $v\in Q_0$.
\end{lemma}

\begin{proof}
This follows from the fact that $Q$ can be identified with a McKay quiver of a group action on $\mathbb{C}^3$. To be more specific, the quiver $Q$ satisfies the following properties:
\begin{enumerate}
    \item[(a)] For each vertex $i$, there are exactly $3$ arrows starting at $i$ which are labeled with $x,y,z$; and there are exactly $3$ arrows ending at $i$, also labeled with $x,y,z$.
    \item[(b)] For each vertex $i$, the target of any non-trivial composition of arrows starting at $i$ depends only on the numbers of arrows labeled with $x,y,z$.
\end{enumerate}
Thus given any plane partition $\pi$ with a $Q$-multi-coloring $K$, we can define $\pi_v$ to be 
\[
\pi_v=\big\{(i,j,k)\in\pi :K_v(i,j,k)\in K(i,j,k)\big\}.
\]
It is easy to check that $\pi_v$ is a plane partition and the union of $\pi_v$'s is $\pi$.
\end{proof}

\begin{corollary}\label{cor:multi-colored}
We have
\[
Z^Q_{\PL}(t_0,t_1,t_2,t_3,t_4)=\prod_{i\in\mathbb{Z}/5\mathbb{Z}}\,Z^{\mu_5(1,1,3)}_{\PL}(t_i,t_{i+1},t_{i+2},t_{i+3},t_{i+4}).
\]
\end{corollary}

This reduces the computation of numbers of $Q$-multi-colored plane partitions to the ones for $\mu_5(1,1,3)$-colored plane partitions.

\begin{remark}
Unfortunately, the signs in the DT invariants $Z^{Q,W}$ and $Z^{[\mathbb{C}^3/\mu_5]}$ do not agree. That is,
\[
Z^{Q,W}(t_0,t_1,t_2,t_3,t_4)\neq\prod_{i\in\mathbb{Z}/5\mathbb{Z}}\,Z^{[\mathbb{C}^3/\mu_5]}(t_i,t_{i+1},t_{i+2},t_{i+3},t_{i+4}),
\]
and there is no obvious modification (e.g.\,changes of variables) to equalize them.
\end{remark}

\begin{remark}
Both $Z^{Q,W}$ and $Z^{[\mathbb{C}^3/\mu_5]}$ are DT invariants of the quiver $(Q,W)$ with potential, with different framing vectors. To put it another way, they are DT invariants of the same Calabi--Yau-3 category with different stability conditions. Thus there should be a formula (``wall-crossing'') connecting two series $Z^{Q,W}$ and $Z^{[\mathbb{C}^3/\mu_5]}$. This can be achieved by, for example, Joyce--Song's generalized DT invariants \cite{JS12}. However, it does not reduce to a simple formula (at least not obvious to us) due to the fact that the Euler pairing ${\chi}_Q(-,-)$ of the quiver $Q$ is not trivial. Hence the formula in \cite[Corollary 7.24]{JS12} does not hold.
\end{remark}

Here we write down the series $Z^{Q,W}$ up to degree $5$. W denote $\tt^{(a_0,\ldots,a_4)}=\sum_i t_i^{a_0}\cdots t_{i+4}^{a_4}$:
\[
\begin{split}
    Z^{Q,W}(\tt)=1 &+\tt^{(1,0,0,0,0)}+3\tt^{(1,1,0,0,0)}-2\tt^{(1,0,1,0,0)}\\
    & +3\tt^{(1,2,0,0,0)}+\tt^{(2,0,1,0,0)}-8\tt^{(1,1,1,0,0)}+8\tt^{(1,1,0,1,0)}\\
    & +\tt^{(1,3,0,0,0)}+3\tt^{(2,1,1,0,0)}-12\tt^{(1,2,1,0,0)}+7\tt^{(1,1,2,0,0)}\\
    & -12\tt^{(1,2,0,1,0)}+5\tt^{(1,1,0,2,0)}-34\tt^{(1,1,1,1,0)}\\
    & -3\tt^{(2,2,1,0,0)}-4\tt^{(2,1,2,0,0)}-6\tt^{(1,3,1,0,0)}+18\tt^{(1,2,2,0,0)}-2\tt^{(1,1,3,0,0)}\\
    & +8\tt^{(1,3,0,1,0)}+10\tt^{(1,2,0,2,0)}+20\tt^{(2,1,1,1,0)}+56\tt^{(1,2,1,1,0)}+35\tt^{(1,1,2,1,0)}\\
    &-54\tt^{(1,1,1,2,0)}-171t_0t_1t_2t_3t_4+O((t_0,t_1,t_2,t_3,t_4)^6)
\end{split}
\]
Take $t=t_0=t_1=t_2=t_3=t_4$, we obtain
\[
Z^{Q,W}(t)=1+5t+5t^2+20t^3-210t^4-131t^5+O(t^6).
\]
We conclude that
\[
\begin{split}
    Z^{\sh{A}}(t) & = \Big(Z^{Q,W}(t)\Big)^{10}\Big( M(-t^5)\Big)^{-50} \\
    & = 1 + 50t + 1175t^2 + 17450t^3 + 184275t^4 + 1450740t^5 + O(t^6).
\end{split}
\]

\subsection{Proof of Theorem~\ref{thm:DT-to-part}}\label{pf:part}

We mainly follow the same method in the computation of DT invariants on $\mathbb{C}^3$ in \cite{BF08}. The path algebra $\mathbb{C}Q$ is
\[
\mathbb{C}Q=\mathbb{C}\langle e,u,v,w\rangle\big/(e^5-1, eu-que, ev-qve, ew-q^3we).
\]
There is a standard $T=(\mathbb{C}^*)^3$-action on $\mathbb{C}Q$ given by
\[
(\lambda_1,\lambda_2,\lambda_3)\cdot(u,v,w)=(\lambda_1 u,\lambda_2 v,\lambda_3 w).
\]
Let $T_0\subset T$ be the sub-torus defined by $\lambda_1\lambda_2\lambda_3=1$, then $T_0$ fixes the potential $W$. Thus it gives a $T_0$-action on $\Jac(Q,W)$.

The Hilbert schemes $\Hilb^{\bullet}(Q,W)$ parameterize quotients of $\Jac(Q,W)$, which are equivalent to left ideals of $\Jac(Q,W)$. Therefore the $T_0$-action on $\Jac(Q,W)$ induces a $T_0$-action on $\Hilb^{\bullet}(Q,W)$. The following is a generalization of \cite[Lemma 4.1]{BF08} 

\begin{lemma}
For each dimension vector $\dd$, there is a one-to-one correspondence between $T_0$-fixed points of $\Hilb^{\dd}(Q,W)$ and $Q$-multi-colored plane partitions with dimension vector $\dd$.
\end{lemma}

\begin{proof}
Since $e_i$'s are idempotent, any ideal in $\Jac(Q,W)$ is generated by polynomials of the form $e_i f(x,y,z)$ for some $i$ and $f(u,v,w)$. We want to show that any $T_0$-invariant ideal can be generated by monomials. We remark that the proof of \cite[Lemma 4.1]{BF08} uses Hilbert's Nullstellensatz, hence it does not apply directly here. However, $\Jac(Q,W)$ contains a $T_0$-invariant subring $\mathbb{C}[u^5,v^5,w^5]$. We take $I_0=I\cap\mathbb{C}[u^5,v^5,w^5]$ is a $T_0$-invariant ideal in $\mathbb{C}[u^5,v^5,w^5]$. Therefore $I_0$ is a monomial ideal, and particularly, there exists $n$ such that $u^{5n},v^{5n},w^{5n}\in I_0\subset I$.

Now, $I$ is generated by eigenvectors of $T_0$, which are polynomials of the form $m(u,v,w)g(uvw)e_i$ for some monomial $m$, $g\in\mathbb{C}[t]$ with $g(0)\neq 0$. Suppose $m(u,v,w)g(uvw)e_i\in I$. We write
\[
g(uvw)=a_0+a_r(uvw)^r+\ldots
\]
Then $a_r(uvw)^rg(uvw)$
\[
\begin{split}
    &\left(g(uvw) - \frac{a_r}{a_0}(uvw)^r g(uvw)\right)m(u,v,w)e_i \\
    =&\,\big(a_0+\tilde{a}_{2r}(uvw)^{2r}+\ldots\big)m(u,v,w)e_i\in I
\end{split}
\]
Repeating this process, we get $(a_0+c(uvw)^{N}+\ldots)m(u,v,w)e_i\in I$ for some $N\geq 5n$, then $m(u,v,w)e_i\in I$. This shows that $I$ is a monomial ideal.

For a monomial ideal $I$ in $\Jac(Q,W)$, we associate a $Q$-multi-colored plane partition $\pi$ as follows: we define
\[
\pi=\big\{(i,j,k): e_{\ell}u^i v^j w^k\notin I\text{ for some }\ell\big\}
\]
and a $Q$-multi-coloring
\[
K(i,j,k)=\big\{\ell: e_{\ell}u^iv^jw^k\notin I\big\}.
\]
We left the details to the reader to check that this indeed defines a $Q$-multi-colored plane partitions. Also, we can associate any $Q$-multi-colored plane partition to a monomial ideal in the obvious way. To compare the dimension vectors, for any (monomial) ideal $I$, there is a natural decomposition
\[
I=e_0I\oplus e_1I\oplus e_2I\oplus e_3I\oplus e_4I
\]
as vector spaces, and the dimension vector $\dd$ of the module induced by $I$ is given by
\[
d_i=\dim_{\mathbb{C}} \big(e_i\Jac(Q,W)\big)/\big(e_iI\big),
\]
which agrees with the dimension vector of the associated multi-colored plane partition.
\end{proof}

\begin{corollary}
For any dimension vector $\dd$, we have
\[
\chi\big(\Hilb^{\dd}(Q,W)\big)= n_Q(\dd).
\]
\end{corollary}

Now we are ready to finalize the proof.

\begin{proof}[Proof of Theorem~\ref{thm:DT-to-part}.]
Recall that $\Hilb^{\dd}(Q,W)$ is the critical locus of the function $\Tr(W)$ on the smooth scheme $\Hilb^{\dd}(\mathbb{C}Q)$. Since $T_0$ acts on $\mathbb{C}Q$ and $W\in\mathbb{C}Q$ is $T_0$-invariant, the torus $T_0$ acts on $\Hilb^{\dd}(\mathbb{C}Q)$ and the function $\Tr(W)$ is $T_0$-invariant. By \cite[Proposition 3.3]{BF08}, the Behrend function $\nu$ of $\Hilb^{\dd}(Q,W)$ is equal to $(-1)^{m}$ on $T_0$-fixed points, where $m$ is the dimension of $\Hilb^{\dd}(\mathbb{C}Q)$, and hence
\[
\chi_{\vir}\big(\Hilb^{\dd}(Q,W)\big)=(-1)^m\chi\big(\Hilb^{\dd}(Q,W)\big)=(-1)^m\,n_Q(\dd).
\]

The Hilbert scheme $\Hilb^{\dd}(\mathbb{C}Q)$ is the fine moduli space of framed representations of $Q$ with framing vector $(1,\ldots,1)$, that is,
\[
\Hilb^{\dd}(\mathbb{C}Q)=\Big(\prod_{(a:i\to j)\in Q_1}\Hom(\mathbb{C}^{d_i},\mathbb{C}^{d_j})\times\prod_i\mathbb{C}^{d_i}\Big)\sslash \prod_{i\in Q_0}\GL_{d_i}(\mathbb{C})
\]
which has dimension
\[
    m=\dim\Hilb^{\dd}(\mathbb{C}Q) =\sum_{a:i\to j} d_id_j + \sum_{i} d_i - \sum_{i} d_i^2 = |\dd|-\langle \dd,\dd\rangle_Q.
\]
\end{proof}

\bibliographystyle{plain}

\begin{thebibliography}{10}

\bibitem{AZ94}
M.~Artin and J.~J. Zhang.
\newblock Noncommutative projective schemes.
\newblock {\em Adv. Math.}, 109(2):228--287, 1994.

\bibitem{Beh09}
Kai Behrend.
\newblock Donaldson-{T}homas type invariants via microlocal geometry.
\newblock {\em Ann. of Math. (2)}, 170(3):1307--1338, 2009.

\bibitem{BF08}
Kai Behrend and Barbara Fantechi.
\newblock Symmetric obstruction theories and {H}ilbert schemes of points on
  threefolds.
\newblock {\em Algebra Number Theory}, 2(3):313--345, 2008.

\bibitem{BCY}
Jim Bryan, Charles Cadman, and Ben Young.
\newblock The orbifold topological vertex.
\newblock {\em Adv. Math.}, 229(1):531--595, 2012.

\bibitem{CMPS}
Alberto Cazzaniga, Andrew Morrison, Brent Pym, and Bal\'{a}zs Szendr\H{o}i.
\newblock Motivic {D}onaldson-{T}homas invariants of some quantized threefolds.
\newblock {\em J. Noncommut. Geom.}, 11(3):1115--1139, 2017.

\bibitem{DOS}
Ben Davison, Jared Ongaro, and Bal\'{a}zs Szendr\H{o}i.
\newblock Enumerating coloured partitions in 2 and 3 dimensions.
\newblock {\em Mathematical Proceedings of the Cambridge Philosophical
  Society}, page 1--27.

\bibitem{JS12}
Dominic Joyce and Yinan Song.
\newblock A theory of generalized {D}onaldson-{T}homas invariants.
\newblock {\em Mem. Amer. Math. Soc.}, 217(1020):iv+199, 2012.

\bibitem{Kan15}
Atsushi Kanazawa.
\newblock Non-commutative projective {C}alabi-{Y}au schemes.
\newblock {\em J. Pure Appl. Algebra}, 219(7):2771--2780, 2015.

\bibitem{KS08}
Maxim Kontsevich and Yan Soibelman.
\newblock {Stability structures, motivic Donaldson-Thomas invariants and
  cluster transformations}.
\newblock 2008.

\bibitem{Liu19}
Yu-Hsiang Liu.
\newblock {D}onaldson-{T}homas theory of quantum {F}ermat quintic threefolds
  {I}.
\newblock 2019.

\bibitem{MNOP1}
D.~Maulik, N.~Nekrasov, A.~Okounkov, and R.~Pandharipande.
\newblock Gromov-{W}itten theory and {D}onaldson-{T}homas theory. {I}.
\newblock {\em Compos. Math.}, 142(5):1263--1285, 2006.

\bibitem{MT18}
Davesh Maulik and Yukinobu Toda.
\newblock Gopakumar-{V}afa invariants via vanishing cycles.
\newblock {\em Invent. Math.}, 213(3):1017--1097, 2018.

\bibitem{PT09}
R.~Pandharipande and R.~P. Thomas.
\newblock Curve counting via stable pairs in the derived category.
\newblock {\em Invent. Math.}, 178(2):407--447, 2009.

\bibitem{GAGA}
Jean-Pierre Serre.
\newblock G\'{e}om\'{e}trie alg\'{e}brique et g\'{e}om\'{e}trie analytique.
\newblock {\em Ann. Inst. Fourier (Grenoble)}, 6:1--42, 1955/56.

\bibitem{Sze08}
Bal\'{a}zs Szendr\H{o}i.
\newblock Non-commutative {D}onaldson-{T}homas invariants and the conifold.
\newblock {\em Geom. Topol.}, 12(2):1171--1202, 2008.

\bibitem{Tho00}
R.~P. Thomas.
\newblock A holomorphic {C}asson invariant for {C}alabi-{Y}au 3-folds, and
  bundles on {$K3$} fibrations.
\newblock {\em J. Differential Geom.}, 54(2):367--438, 2000.

\bibitem{Toda18}
Yukinobu Toda.
\newblock Moduli stacks of semistable sheaves and representations of
  {E}xt-quivers.
\newblock {\em Geom. Topol.}, 22(5):3083--3144, 2018.

\bibitem{You10}
Benjamin Young.
\newblock Generating functions for colored 3{D} {Y}oung diagrams and the
  {D}onaldson-{T}homas invariants of orbifolds.
\newblock {\em Duke Math. J.}, 152(1):115--153, 2010.
\newblock With an appendix by Jim Bryan.

\end{thebibliography}

\end{document}